\documentclass[preprint,3p,times]{elsarticle}
\usepackage{lineno,hyperref}
\usepackage{amsmath}
\usepackage{mathrsfs}
\usepackage{tikz}
\usepackage{slashbox}
\usepackage{hyperref}
\usepackage{listings}
\usepackage{multirow}
\usepackage{amssymb}
\usepackage{amsthm}
\usepackage{epstopdf}
\usepackage{algorithm}
\usepackage{algpseudocode}
\usepackage{subfigure}
\usepackage{tabularx,booktabs}
\usepackage{caption}
\usepackage{verbatim}
\usepackage{graphicx}

\biboptions{sort&compress}
\theoremstyle{plain}
\newtheorem{theorem}{Theorem}[section]
\newtheorem{lemma}{Lemma}[section]
\newtheorem{proposition}{Proposition}[section]
\newtheorem{corollary}{Corollary}[section]
\theoremstyle{definition}
\newtheorem{definition}{Definition}[section]

\newtheorem{example}{Example}[section]
\newtheorem{remark}{Remark}[section]

\begin{document}

\begin{frontmatter}

\title{The generalizations of fuzzy monoids and vague monoids}
\author{Wei Li \fnref{label1}}
\author{Haohao Wang\fnref{label1}}
\author{Yuanhao Liu\fnref{label1}}
\author{Bin Yang\fnref{label1}\corref{cor1}}
\ead{binyang0906@whu.edu.cn, binyang0906@nwsuaf.edu.cn}
\address[label1]{College of Science, Northwest A \& F University, Yangling 712100, PR China}
\cortext[cor1]{Corresponding author.}

\begin{abstract}
In this paper, we present the fuzzy monoids and vague monoids by using aggregation operators. The unit interval with a $t$-norm or a $t$-conorm is a special monoid, so we mainly talk about fuzzy subsets of monoids. Firstly, the classification of fuzzy sets based on some special aggregation operators is discussed. At the same time, we give two basic propositions about submonoids of $t$-norm and $t$-conorm. The fuzzification by uninorm and nullnorm are denoted and some properties can be drawn. Next, we briefly present fuzzy subsets on lattice. Finally, the vague monoids on aggregation operators are redefined and further consider the special cases of uninorms and nullnorms.
\end{abstract}


\end{frontmatter}


\section{Introduction}
Aggregation operation is an important tool in fuzzy mathematics, which can fuse a given set of data into a representative value. There are many types of aggregation operations, which require different choices in different situations. In fuzzy logic, we usually use the triangular norms ($t$-norms for short) and the triangular conorms ($t$-conorms for short), as two kinds of fuzzy logic operations, which play a crucial role in fuzzy sets theory \cite{Zadeh1965Fuzzysets,Zimmermann1991FuzzySet}. To further enrich the properties of the aggregation operations, Yager and Rybalov \cite{Yager1996Uninormaggregation} proposed the concept of uninorm. The identity element of uninorm can take any number in the unit interval, not just zero and one in the case of $t$-norms and $t$-conorms. In addition, Calvo et al.\cite{Calvo2001Thefunctionalequations} introduced the notion of nullnorm in 2001.
Both of uninorms and nullnorms are a generalization of $t$-norms and $t$-conorms, and on the other side, there exists close connection between them.

The fuzzification or relaxation of logical connectives can be effective in solving practical problems such as imprecision, lack of accuracy, or the presence of noise. Considering that many scholars are very interested in fuzzified algebraic structures, including groups, rings, actions\cite{Boixader2018Fuzzyactions, RosenfeldFuzzy1971subgroups}.

Boixader and Recasens \cite{D.BoixaderVagueandfuzzy2021} fuzzified the concept of submonoid and further explored the fuzzy $t$-subnorm and fuzzy $t$-subconorm, which can deal with imprecision effectively. Refer to the method of fuzzy monoid in  \cite{D.BoixaderVagueandfuzzy2021},
 we will study the fuzzification of submonoid by generalizing the $t$-norm and $t$-conorm to aggregation operators, named $A$-fuzzy submonoid of set $M$, where $A$ is any aggregation operation.
 Further, the corresponding conclusions are given when the aggregation function $A$ takes the uninorm $U$ and nullnorm $F$.
 The second part of this paper, we explore the concept of $A$-vague binary operations and $A$-vague monoids. And then, further the homomorphic relationship between $A$-vague monoids and $A$-fuzzy monoids is constructed.



The rest of this paper is arranged as follows. In section \ref{Preliminaries}, we review some basic concepts and results related to aggregation functions, $t$-norms, $t$-conorms, uninorms, nullnorms and fuzzy monoids.
In section \ref{The generation of fuzzy submonoid}, we give the definition of fuzzy submonoid which generates by aggregation operators, and use it on uninorms and nullnorms.
 And then, we apply the concept of fuzzy submonoid to monoid $([0, 1],T)$ and $([0, 1], S)$, where $T$ and $S$ are $t$-norm and $t$-conorm respectively. Further, what we call $U$-fuzzy $t$-subnorms and $U$-fuzzy $t$-subconorms (resp. $F$-fuzzy $t$-subnorms and $F$-fuzzy $t$-subconorms) can be obtained later, and the above conclusions can be generalized to lattice.
In section \ref{Vague monoids}, the vague binary operators and vague monoids are introduced and the concept of homomorphic mapping between vague monoids and fuzzy monoids are given. Next,  we obtain the corresponding conclusions when aggregation function $A$ takes the uninorm $U$ and nullnorm $F$.
Finally, conclusions on our research are given in section \ref{Concluding remarks}.


\section{Preliminaries}\label{Preliminaries}
	In this section, we briefly recall some fundamental concepts and definitions about aggregation operations, such as $t$-norms, $t$-conorms, uninorms and nullnorms, which shall be needed in the sequel.
	\begin{definition}\label{aggregation function}
		A function $H: [0,1]^n\to [0,1]$ is a $n$-ary aggregation operation, which satisfies the following conditions:
		\begin{enumerate}[(1)]
			\item
			$A(0,0,\cdots, 0)=0$ and 	$A(1,1,\cdots, 1)=1$;
			\item
			If $x_i\le y$ for all $i\in\{1, 2,,\cdots,n\}$, then
			\begin{align*}
				A(x_1,\cdots,x_{i-1}, x_i, x_{i+1},\cdots, x_n)\le A(x_1,\cdots,x_{i-1}, y, x_{i+1},\cdots, x_n).
			\end{align*}
		\end{enumerate}
	\end{definition}

For more details about aggregate functions, please refer to  \cite{BeliakovAggregationFunctions2007,GomezAdiscussion2004, KomornikovaAggregationoperators2001}.

	\begin{definition}\label{d:t-norm}\cite{E.P.KlementTriangularNorms2000}
		A $t$-norm is a binary operation $T:[0,1]^2\to [0,1]$, for all $x,y,z\in [0,1]$, which satisfies the following conditions:
		\begin{enumerate}[(1)]
			\item
			$T(x,y)=T(y,x)$;
			\item
			$T(T(x,y),z)=T(x,T(y,z))$;
			\item
			$T$ is non-decreasing in each argument;
			\item
			$T(x,1)=x.$
		\end{enumerate}	 	
	\end{definition}	

\begin{example}\label{e:t-norm}\cite{E.P.KlementTriangularNorms2000}
	The following shows several common $t$-norms:\\
	\begin{itemize}
		\item
		Minimum: $T_M(x, y) = \textnormal{min}(x, y)$;
		\item
		Product: $T_P(x, y) = xy$;
		\item
		{\L}ukasiewicz $t$-norm: $T_L(x, y) = \textnormal{max}(x+y-1,0)$;
		\item
		Drastic product:
		$T_D(x,y) =\left\{
		\begin{aligned}
		&0&if\ x,y\in [0,1), \\
		&\textnormal{min}(x,y)&otherwise.
		\end{aligned}
		\right.$
	\end{itemize}
\end{example}

\begin{definition}\label{d:t-conorm}\cite{E.P.KlementTriangularNorms2000}
	A $t$-conorm is a binary operation $S:[0,1]^2\to [0,1]$, for all $x,y,z\in [0,1]$, which satisfies the following conditions:
	\begin{enumerate}[(1)]
		\item
		$S(x,y)=S(y,x)$;
		\item
		$S(S(x,y),z)=S(x,S(y,z))$;
		\item
		$S$ is non-decreasing in each argument;
		\item
		$S(x,0)=x.$	
	\end{enumerate}	
\end{definition}

\begin{example}\label{e:t-conorm}\cite{E.P.KlementTriangularNorms2000}
	The following shows several common $t$-conorms:\\
	\begin{itemize}
		\item
		Maximum: $S_M(x, y) = \textnormal{max}(x, y)$;
		\item
		Probabilistic sum: $S_P(x, y) = x+y-xy$;
		\item
		{\L}ukasiewicz $t$-conorm: $S_L(x, y) = \textnormal{min}(x+y,1)$;
		\item
		Drastic sum:
		$S_D(x,y) =\left\{
		\begin{aligned}
		&1&if\ x,y\in (0,1], \\
		&\textnormal{max}(x,y)&otherwise.
		\end{aligned}
		\right.$
	\end{itemize}
\end{example}

\begin{definition}\cite{E.P.KlementTriangularNorms2000}
	An element $x\in [0,1]$ is called an idempotent element for a $t$-norm $T$ (resp. $t$-conorm $S$)if and only if $T(x,x)=x$ (resp. $S(x,x)=x$).
	$0$ and $1$ which are idempotent elements for $T$ and $S$, respectively, called trivial idempotent elements of $T$ and $S$, the rest are non-trivial idempotent element of $T$ and $S$.
	The set of idempotent elements of $T$ (resp. $S$) is noted as $Ide(T)$ (resp. $Ide(S)$).
\end{definition}

\begin{remark}
	A continuous $t$-norm $T$ (resp. $t$-conorm $S$) is Archimedean if and only if $Ide(T) = \{0, 1\}$ (resp. $Ide(S) = \{0, 1\}$).
\end{remark}

\begin{example}
	Notice $T_L$, $T_P$ are Archimedean $t$-norm, and $T_M$ is not.
\end{example}


\begin{proposition}\label{p:additive generator}\cite{E.P.KlementTriangularNorms2000}
	A continuous $t$-norm $T$ is Archimedean if and only if there exists a continuous and strictly decreasing function $f : [0, 1] \to [0,\infty]$ with $f(1) = 0$ such that
	\begin{align*}
		T(x,y) = f^{[-1]}(f(x)+f(y)),
	\end{align*}
	where $f^{[-1]}$ is the pseudo inverse of $f$, denoted by
	$$f^{[-1]}(x)=\left\{
	\begin{aligned}
	 &f^{-1}(x), &if\ x\in[0,f(0)], \\
	 &0, &otherwise.
	\end{aligned}
	\right.$$
	$f$ is an additive generator of $T$ and two additive generators of the same $t$-norm differ only by a positive constant multiple.
\end{proposition}

\begin{proposition}\label{p:additive generator}\cite{E.P.KlementTriangularNorms2000}
	A continuous $t$-conorm $S$ is Archimedean if and only if there exists a continuous and strictly decreasing function $g : [0, 1] \to [0,\infty]$ with $g(0) = 0$ such that
	\begin{align*}
	S(x,y) = g^{[-1]}(g(x)+g(y)),
	\end{align*}
	where $g^{[-1]}$ is the pseudo inverse of $g$, denoted by
	$$g^{[-1]}(x)=\left\{
	\begin{aligned}
	&g^{-1}(x), &if\ x\in[0,g(1)], \\
	&1, &otherwise.
	\end{aligned}
	\right.$$
	$g$ is an additive generator of $S$ and two additive generators of the same $t$-conorm differ only by a positive constant multiple.
\end{proposition}

\begin{definition}\label{def3}\cite{Yager1996Uninormaggregation}
	A uninorm is a binary function $U : [0, 1]^2 \rightarrow [0, 1]$, for all $x, y, z \in [0, 1]$, which satisfies the following conditions :
	\begin{enumerate}[(U1)]
		\item
		$U(x, y) = U(y, x)$;
		\item
		$U(U(x, y), z) = U (x,U(y, z))$;
		\item
    	$U$ is non-decreasing in each argument;
		\item
		$U(x, e) = x, e\in [0,1]$.	
	\end{enumerate}	
\end{definition}

\begin{remark}
	In particular, a uninorm is a $t$-norm when $e=1$ and a $t$-conorm when $e=0$. For any $e\in (0,1)$, uninorm is equivalent to a $t$-norm in $[0,e]^2$, a $t$-conorm in $[e,1]^2$, and its values in $A(e) = [0,e)\times (e,1] \cup (e,1] \times [0,e)$ have the following form:
	$$\textnormal{min}(x,y)\leqslant U(x,y)\leqslant \textnormal{max}(x,y).$$
\end{remark}

\begin{remark}
	A uninorm $U$ is called conjunctive if $U(1, 0)=0$ and disjunctive if $U(1,0)=1$. A conjunctive (resp. disjunctive) uninorm $U$ is said to be locally internal on the boundary if it
	satisfies $U (1, x) \in \{1, x\}$ (resp. $U(0, x) \in \{0, x\}$) for all $x \in [0,1]$.
	In addition, several common classes of uninorms are shown below.
    \begin{itemize}
	\item
	$\mathscr{U}_{min}$, $\mathscr{U}_{max}$: those given by minimum and maximum in $A(e)$, respectively.
	\item
	$\mathscr{U}_{ide}$: $U(x, x)=x$ for all $x \in [0, 1]$.
	\item
	$\mathscr{U}_{rep}$: those that have an additive generator.
	\item
	$\mathscr{U}_{cos}$: Continuous in the open square $(0,1)^2$.
   \end{itemize}
\end{remark}

\begin{theorem}\cite{Fodor1997Structure}
	Let $U:[0, 1]^2\rightarrow [0, 1]$ be a uninorm with identity element $e \in (0,1)$. Then the sections $x \hookrightarrow U(x, 1)$ and $x \hookrightarrow U(x, 0)$ are continuous in each point except perhaps for $e$ if and only if $U$ is given by one of the following formulas.
	\begin{enumerate}[(1)]
		\item
		If $U(0, 1)= 0$, then
		$$ U(x,y)=\left\{
		\begin{aligned}
		&eT(\frac{x}{e}, \frac{y}{e}), &if(x,y)\in [0,e]^2.\\
		&e+(1-e)S(\frac{x-e}{1-e},\frac{y-e}{1-e}), &if(x,y)\in [e,1]^2.\\
		&\textnormal{min}(x,y), &if(x,y)\in A(e).
		\end{aligned} \right. $$
		The set of uninorms as above will be denoted by $\mathscr{U}_{min}$.
		\item
		If $U(0, 1)=1$, then $\mathscr{U}_{max}$ has the same structure by changing minimum by maximum in $A(e)$.
	\end{enumerate}
\end{theorem}

\begin{theorem}\cite{D.Ruiz-AguileraSomeremarks2010}
	$U$ is an idempotent uninorm with identity element $e \in [0, 1]$ if and only if there exists a non-increasing function $g :[0, 1] \to [0, 1]$, symmetric with respect to the main diagonal, with $g(e) = e$, such that
	$$U(x,y)\!=\! \left\{
	\begin{aligned}
	&\textnormal{min}(x,y),
	&if\ y<g(x)\ or\ (y=g(x)\ and\ x<g(g(x))).\\
	&\textnormal{max}(x,y),
	&if\ y>g(x)\ or\ (y=g(x)\ and\ x>g(g(x))).\\
	&\textnormal{min}(x,y)\ or\ \textnormal{max}(x,y),
	&if\ y=g(x)\ and\ x=g(g(x)).\\
	\end{aligned} \right. $$
	\normalfont
	being commutative in the points $(x, y)$ such that $y = g(x)$ with $x = g(g(x))$.
\end{theorem}

\begin{theorem}\cite{Fodor1997Structure}
	A binary operation $U:[0, 1]^2 \to [0, 1]$ is a representable uninorm if and only if there exists a continuous strictly increasing function $h:[0, 1] \to [-\infty,+\infty]$ with $h(0) = -\infty$, $h(e) = 0\ (e\in (0,1))$ and $h(1) = +\infty$ such that
	\begin{center}
		$U(x,y)=h^{-1}(h(x)+h(y)),$
	\end{center}
	for all $(x, y) \in [0, 1]^2 \setminus \{(0, 1),(1, 0)\}$ and $U(0, 1)=U(1,0)\in\{0,1\}$. The function $h$ is usually called an additive generator of $U$.
\end{theorem}

\begin{theorem}\cite{S.K.HuThestructure2001, D.RuizDistributivityand2006}
	Let $U$ be a uninorm continuous in $(0,1)^2$ with identity element $e\in (0,1)$. Then either one of the following cases is satisfied.
	
	(a) There exist $u \in [0, e]$, $\lambda \in [0,u]$, two continuous $t$-norms $T_1$ and $T_2$ and a representable uninorm $R$ such that $U$ can be represented as
	\begin{equation}\label{cosmin}
	U(x,y)\!=\! \left\{
	\begin{aligned}
	&\lambda T_1(\frac{x}{\lambda},\frac{y}{\lambda}),
	&if\ x,y\in [0,\lambda].\\
	&\lambda+(u-\lambda)T_2(\frac{x-\lambda}{u-\lambda},\frac{y-\lambda}{u-\lambda}),
	&if\ x,y\in [\lambda,u].\\
	&u+(1-u)R(\frac{x-u}{1-u},\frac{y-u}{1-u}),
	&if\ x,y\in [u,1].\\
	&1,
	&if\ \textnormal{min}(x,y)\in (\lambda,1]\ and\ \textnormal{max}(x,y)=1.\\
	&\lambda \ or\ 1 ,
	&if(x,y)\in\{(\lambda,1),(1,\lambda)\}.\\
	&\textnormal{min}(x,y),
	&elsewhere.
	\end{aligned} \right.
	\end{equation}
	
	(b) There exist $v \in (e, 1]$, $\omega \in [v,1]$, two continuous $t$-conorms $S_1$ and $S_2$ and a representable uninorm $R$ such that $U$ can be represented as
	\begin{equation}\label{cosmax}
	U(x,y)\!=\! \left\{
	\begin{aligned}
	&v R(\frac{x}{v},\frac{y}{v}),
	&if\ x,y\in [0,v].\\
	&v+(\omega-v)S_1(\frac{x-v}{\omega-v},\frac{y-v}{\omega-v}),
	&if\ x,y\in [v,\omega].\\
	&\omega+(1-\omega)S_2(\frac{x-\omega}{1-\omega},\frac{y-\omega}{1-\omega}),
	&if\ x,y\in [\omega,1].\\
	&0,
	&if\ \textnormal{max}(x,y)\in [0,\omega)\ and\ \textnormal{min}(x,y)=0.\\
	&\omega \ or\ 0,
	&if(x,y)\in\{(0,\omega),(\omega,0)\}.\\
	&\textnormal{max}(x,y),
	&elsewhere.
	\end{aligned} \right.
	\end{equation}
\end{theorem}

$\mathscr{U}_{cos}$ is used to denote the class of all uninorms continuous in $(0,1)^2$. A uninorm as in (\ref{cosmin}) will be defined by  $U\equiv \langle T_1,\lambda,T_2,u,(R,e)\rangle _{cos,min}$, $\mathscr{U}_{cos,min}$ for short. Similarly, a uninorm as in (\ref{cosmax}) will be defined by $$U\equiv \langle(R,e),v,S_1,\omega,S_2\rangle_{cos,max}, \mathscr{U}_{cos,max}$$ for short.

\begin{definition}\label{def4}\cite{Calvo2001Thefunctionalequations, M.Mast-Operators1999}
	A nullnorm is a binary function $F : [0, 1]^2 \rightarrow [0, 1]$, for all $x, y, z \in [0, 1]$, which satisfies the following conditions:
	\begin{enumerate}[(F1)]
		\item
		$F(x, y) = F(y, x)$;
		\item
		$F(F(x, y), z) = F(x,F(y, z))$;
		\item
		$F$ is non-decreasing in each place;
		\item
		There exists an absorbing element $k \in [0, 1],\, F(k, x) = k$ and the following statements hold.
		\begin{enumerate}[(i)]
			\item
			$F(0, x) = x$ for all $x \le k$;
			\item
		    $F(1, x) = x$ for all $x \ge k$.
		\end{enumerate}	
	\end{enumerate}	
In general, $k$ is always given by $F(0,1)$.	
\end{definition}

  Denote $X$ as the universe, then a fuzzy subset $A$ of $X$ can be represented as
  	$\mu_A: X\to [0,1],$
  	where for any $x \in X$, $\mu_A (x)$ is denoted as the membership function of $x$ in fuzzy set $A$. The family of all fuzzy sets of $X$ will be defined as $\mathscr{F}(X)$.
    Further, let $L$ be a nonempty poset, then $(L;\land, \lor)$ is called a lattice, if $x\land y$ and $x\lor y$ always exist for all $x,y\in L$, where the join and meet operations are denoted by ``$\vee$" and ``$\wedge$" in $L$, respectively. Then the concept of fuzzy subset on $L$ will be given.
	
\begin{definition}
	Let $(L,\land,\lor)$ be a lattice. A fuzzy subset $A$ of $L$ can be represented as follows:
	$$\mu_A: L\to [0,1],$$
	where for any $x \in L$, $\mu_A (x)$ is denoted as the membership function of $x$ in fuzzy set $A$. The family of all fuzzy sets of $L$ will be defined as $\mathscr{F}(L)$.
\end{definition}
\begin{definition}\cite{D.BoixaderVagueandfuzzy2021}
		A monoid ($H$, $\ast$ ) consists of a set $H$ with a binary operation $\ast$ : $H^2\to H$, which has identity element and associativity.
\end{definition}

\begin{definition}\cite{D.BoixaderVagueandfuzzy2021,RosenfeldFuzzy1971subgroups}\label{t_fuzzy_submonoid}
	Let $(H, \ast)$ be a monoid with identity element $e$, $T$ a $t$-norm and $\sigma$ a fuzzy subset of $H$. Then $\sigma$ is a $T$-fuzzy
	submonoid of $H$ is equivalent to the following conditions.
	\begin{itemize}
		\item
		$T(\sigma(a),\sigma(b)) \le \sigma(a\ast b)\quad  \forall a, b \in H.$
		\item
		$\sigma (e) = 1.$
	\end{itemize}
\end{definition}

\section{The generalization of fuzzy submonoid}\label{The generation of fuzzy submonoid}
\subsection{$A$-fuzzy submonoid}
Refer to the Definition~\ref{t_fuzzy_submonoid}, we propose the concept of fuzzy submonoid based on aggregation operator $A$.
\begin{definition}
	Let $A$ be an aggregation operator, ($M$, $\circ$) a monoid with identity element $e$ and $\sigma$ a fuzzy subset of $M$. Then $\sigma$ is a $A$-fuzzy submonoid of $M$ if and only if $\sigma$ satisfies the following two conditions:
	\begin{enumerate}[(1)]
		\item 	 $A(\sigma(x_1),\cdots,\sigma(x_n))\leq \sigma(x_1 \circ \cdots \circ x_n), \forall x_1,\cdots,x_n \in M; $
		\item	$\sigma(e) = 1$.
	\end{enumerate}
\end{definition}

\begin{proposition}
	Let $(M,\circ)$ be a monoid and $\sigma$ an $A$-fuzzy submonoid of $M$. Then the core $H$ of $\sigma$ (i.e., the set of elements $x$ in $M$ such that $\sigma(x) = 1$) is a submonoid of $M$.
\end{proposition}

\begin{proof}
	The identity element of $H$ obviously existed and the associativity is inherited.\\
	For all  $x,y \in H$, we have $1=A(\sigma(x),\sigma(y),\sigma(e),\cdots,\sigma(e))$.
	Since $\sigma$ is an $A$-fuzzy submonoid of $M$. Then we have that
	$$A(\sigma(x),\sigma(y),\sigma(e),\cdots,\sigma(e)) \leq \sigma(x \circ y \circ e \circ \cdots \circ e)=\sigma(x\circ y).$$
	Hence, $\sigma(x\circ y) = 1$, that is, $x \circ y \in H$.
\end{proof}

\begin{example}\label{e:min_max}
Let $A_{min}(x_1,\cdots,x_n)=\textnormal{min}(x_1,\cdots,x_n)$, fuzzy subset $\sigma_1(x)=x$, which is the $A_{min}$-fuzzy submonoid  of $([0,1],T_M)$. Similarly, fuzzy subset $\sigma_2(x)=1-x$, which is the $A_{min}$-fuzzy submonoid  of $([0,1],S_M)$.
\end{example}

In fuzzy logic, the unit interval with a $t$-norm or a $t$-conorm is the most important monoid. Thus, we will consider fuzzy submonoids of a given $t$-norm or $t$-conorm.

\begin{definition}
	Let $A$ and $T$ be an aggregation operator and a $t$-norm, respectively. An $A$-fuzzy submonoid of $([0,1],T)$ will be called an $A$-fuzzy $t$-subnorm of $T$.
\end{definition}
\begin{definition}
	Let $A$ and $S$ be an aggregation operator and a $t$-conorm, respectively. An $A$-fuzzy submonoid of $([0,1],S)$ will be called an $A$-fuzzy $t$-subconorm of $S$.
\end{definition}

The following propositions give the more general conclusions of Example~\ref{e:min_max}.

\begin{proposition}
	Let  $A(x_1,\cdots,x_n)=\textnormal{min}(x_1,\cdots,x_n)$. $\sigma$ is an $A$-fuzzy $t$-subnorm of $T_M$ if and only if $\sigma(1)=1$.
\end{proposition}
\begin{proof}
	For all fuzzy subset $\sigma$ of $[0,1]$, one can conclude that
	$$ \textnormal{min}(\sigma(x_1),\cdots,\sigma(x_n))\leq \sigma(\textnormal{min}(x_1,\cdots,x_n)) ,$$
	where $x_1,\cdots, x_n \in [0,1]$.
\end{proof}

\begin{proposition}
	Let  $A(x_1,\cdots,x_n)=\textnormal{min}(x_1,\cdots,x_n)$. $\sigma$ is an $A$-fuzzy t-subconorm of $S_M$  if and only if $\sigma(0)=1$.
\end{proposition}
\begin{proof}
	For all fuzzy subset $\sigma$ of [0,1], one can conclude that
	$$ \textnormal{min}(\sigma(x_1),\cdots,\sigma(x_n))\leq \sigma(\textnormal{max}(x_1,\cdots,x_n)) ,$$
	where $x_1,\cdots,x_n \in [0,1]$.
\end{proof}

\subsection{$U$-fuzzy submonoid}\label{$U$-fuzzy submonoid}
In particular, when the aggregation operations take uninorms, we have the following conclusions.
\begin{definition}
	Let $U$ be a uninorm, ($M$, $\circ$) a monoid with identity element $e$ and $\sigma$ a fuzzy subset of $M$. Then $\sigma$ is a $U$-fuzzy submonoid of $M$ if and only if $\sigma$ satisfies the following conditions:
	\begin{enumerate}[(1)]
		\item 	 $U(\sigma(x),\sigma(y))\leq \sigma(x \circ y), \forall x, y \in M;$
		\item	$\sigma(e) = 1$.
	\end{enumerate}
\end{definition}
\begin{proposition}
	Let $U$ be a uninorm, $(M,\circ)$ a monoid and $\sigma$ a $U$-fuzzy submonoid of $M$. Then the core $H$ of $\sigma$ (i.e., the set of elements $x$ of $M$ such that $\sigma(x) = 1$) is a submonoid of $M$.
\end{proposition}
\begin{proof}
	The identity element of $H$ obviously existed and the associativity is inherited.
	For all $x,y \in H$, we have that
	$$ 1=U(\sigma(x),\sigma(y))\leq\sigma(x\circ y).$$
	Therefore, $x \circ y \in H$.
\end{proof}

\begin{definition}
	A discrete uninorm is a submonoid of a uninorm containing $0$ and $1$.
\end{definition}
\begin{example}
	Let $U_L= \langle T_L,e,S_L \rangle _{\textnormal{min}}, L_{n,m}=\{0,\frac{e}{n},\cdots,e,e+\frac{1-e}{m},\cdots,1 \} $. Then $L_{n,m}$ is a discrete uninorm of $U_L$.
	
\end{example}
\begin{definition}\label{d:U-fuzzy_t-subnorm_T}
	Let $U$ and $T$ be a uninorm and a $t$-norm, respectively. A $U$-fuzzy submonoid of $([0,1],T)$ will be called a $U$-fuzzy $t$-subnorm of $T$.
\end{definition}
\begin{definition}\label{d:U-fuzzy_t-subconorm_S}
	Let $U$ and $S$ be a uninorm and a $t$-conorm, respectively. A $U$-fuzzy submonoid of $([0,1],S)$ will be called a $U$-fuzzy $t$-subconorm of $S$.
\end{definition}
The following theorem gives the necessary and sufficient conditions for $\sigma$ to be a $U$-fuzzy submonoid of $(M,\circ)$.
\begin{theorem}
For a monoid $M$ with identity element $e$, if a uninorm U is disjunctive, then $\sigma$ is $U$-fuzzy submonoid of $(M,\circ)$ if and only if $\sigma \equiv 1$.
\end{theorem}
\begin{proof}
	Since $U$ is disjunctive, i.e., $U(0,1)=1$. Then for all $x\in [0,1]$, we have that $U(x,1)=1$.
	
	If $\sigma$ is $U$-fuzzy submonoid of $([0,1],M)$, then $\sigma(e)=1$ and
	$U(\sigma(x),\sigma(y))\leq \sigma(x\circ y).$
	For all $x \in [0,1]$, taking $y=e$, we can conclude that $1=U(\sigma(x),1)\leq \sigma(x)$.
	
	The other hand, if $\sigma \equiv 1$, then $\sigma$ obviously is $U$-fuzzy submonoid of $([0,1],M)$.
\end{proof}
Further, if we take monoid as $t$-norm and $t$-conorm, the following corollaries hold.
\begin{corollary}
	For a t-norm $T$, if uninorm $U$ is disjunctive, then $\sigma$ is $U$-fuzzy $t$-subnorm of $([0,1],T)$ if and only if $\sigma \equiv 1$.
\end{corollary}

\begin{corollary}
	For a t-conorm $S$, if uninorm $U$ is disjunctive, then $\sigma$ is $U$-fuzzy  $t$-subconorm of $([0,1],S)$ if and only if $\sigma$ $\equiv$ 1.
\end{corollary}

\begin{corollary}
For arbitrary uninorm $U$ in $\mathscr{U}_{max}$, if fuzzy subset $\sigma$ is $U$-fuzzy submonoid of $M$, then $\sigma \equiv 1$.
\end{corollary}

\begin{proposition}\label{decreasing}
	Let $U$ be a uninorm and $B$ be a set, which denoted as $B=\{x\in[0,1]
	\,|\,\sigma(x)\in [e,1]\}$.	
	\begin{equation*}
	U(x,y)=\left\{
	\begin{aligned}
	&eT(\frac{x}{e},\frac{y}{e}),&(x,y)\in[0,e]^2,\\
	&e+(1-e)S_M(\frac{x-e}{1-e},\frac{y-e}{1-e}),&(x,y)\in (e,1]^2,\\
	&\textnormal{min}(x,y),&otherwise.
	\end{aligned}
	\right.
	\end{equation*}
	where $\sigma$ is fuzzy subset. Then $\sigma$ is $U$-fuzzy t-subnorm of $([0,1],T_M)$ if and only if $\sigma$ is decreasing on $B$ and $\sigma(1)=1$.
\end{proposition}

\begin{proof}
	
	\begin{enumerate}[(1)]
		\item
		If $(\sigma(x),\sigma(y))\in [0,e]^2$, then
		\begin{align*}
		U((\sigma(x),\sigma(y))) &=eT(\frac{\sigma(x)}{e},\frac{\sigma(y)}{e})\\
		&\leq \textnormal{min}(\sigma(x),\sigma(y))\\
		& \leq \sigma(\textnormal{min}(x,y)).
		\end{align*}
		\item
		If $(\sigma(x),\sigma(y))\in [e,1]^2$, then
		\begin{align*}
		U(\sigma(x), \sigma(y))=\textnormal{max}(\sigma(x),\sigma(y)) \leq \sigma(\textnormal{min}(x,y)),	
		\end{align*}	
		if and only if
		$x\leq y$, then $\sigma(y)\leq \sigma(x)$.
		\item
		If $(\sigma(x),\sigma(y))\in [0,e]\times [e,1]$, one concludes that
		\begin{align*}
		U(\sigma(x),\sigma(y))=\textnormal{min}(\sigma(x),\sigma(y)) \leq \sigma(\textnormal{min}(x,y)).
		\end{align*}	
	\end{enumerate}
\end{proof}

Then the dual conclusions about $U$-fuzzy $t$-subnorm of $([0,1],S_M)$ can be obtained.
\begin{proposition}\label{increasing}
	Let $U$ be a uninorm and $B$ be a set, which denoted as $B=\{x\in[0,1]\,|\, \sigma(x)\in [e,1]\}$.
	\begin{equation*}
	U(x,y)=\left\{
	\begin{aligned}
	&eT(\frac{x}{e},\frac{y}{e}),&(x,y)\in[0,e]^2,\\
	&e+(1-e)S_M(\frac{x-e}{1-e},\frac{y-e}{1-e}),&(x,y)\in (e,1]^2,\\
	&\textnormal{min}(x,y),&otherwise.
	\end{aligned}
	\right.
	\end{equation*}
	where $\sigma$ is a fuzzy subset. Then $\sigma$ is a $U$-fuzzy $t$-subconorm of $([0,1],S_M)$ if and only if $\sigma$ is increasing on $B$ and $\sigma(0)=1$.
\end{proposition}
\begin{proof}
	It can be proven in a similar way as Proposition \ref{decreasing}.
\end{proof}

\begin{example}
	In particular, the fuzzy subset $\sigma$ is denoted as:
	\begin{equation*}
	\sigma(x)=\left \{
	\begin{aligned}
	&x, &0\leq x< e,\\
	&1, &e\leq x \leq 1.
	\end{aligned}
	\right.	
	\end{equation*}	
	According to Proposition~\ref{decreasing}, $B=\{x\in[e,1]\, |\, \sigma(x)=1 \}$ and $\sigma$ is decreasing on $B$, then $\sigma$ is $U$-fuzzy $t$-subnorm of $([0,1], T_M)$.
\end{example}

Not all fuzzy subset $\sigma$ of a monoid $M$ can find corresponding uninorm $U$ such that  $\sigma$ is a $U$-fuzzy submonoid of $M$. Then the following propositions hold.

\begin{proposition}
	Let $\sigma(x)=x$ be a fuzzy subset, $T$ be a $t$-norm. There is no uninorm $U$ such that $\sigma$ is a $U$-fuzzy $t$-subnorm of $([0,1],T)$.
\end{proposition}

	\begin{proof}
	Assume that there exists uninorm $U$ with identity element $e$ that $\sigma$ is a $U$-fuzzy $t$-subnorm of $([0,1],T)$, then
	$$U(\sigma(x),\sigma(y))\leq \sigma(T(x,y)).$$
	Furthermore, it holds that
	$$U(x,y)\leq T(x,y)\leq \textnormal{min}(x,y),$$
	which contradicts with the case of $x=e, y>e$.
\end{proof}

\begin{proposition}
	Let $\sigma(x)=1-x$ be a fuzzy subset, $S$ be a $t$-conorm. There is no uninorm $U$ that $\sigma$ is a $U$-fuzzy $t$-subconorm of $([0,1],S)$.
\end{proposition}
\begin{proof}
	Assume that there exists uninorm $U$ with identity element $e$ that $\sigma$ is a $U$-fuzzy $t$-subconorm of $([0,1],S)$, then
	$$U(\sigma(x),\sigma(y))\leq \sigma(S(x,y)).$$
	Furthermore, it holds that
	$$U(1-x,1-y)\leq 1-S(x,y)\leq 1-\textnormal{max}(x,y),$$
	which contradicts with the case of $x=1-e, y< 1-e$.
\end{proof}

Next, a sufficient and necessary condition for fuzzy subset $\sigma$ to be $U$-fuzzy $t$-subnorm of $([0, 1], T)$ (resp. $U$-fuzzy $t$-subconorm of $([0, 1], S)$) is explored based on representable uninorms and continuous Archimedean $t$-norms (resp. $t$-conorm).
\begin{proposition}
	Let $U$ be representable uninorms and $T$ be continuous Archimedean $t$-norms. $h$ and $t$  are additive generators of $U$ and $T$, respectively. A fuzzy subset $\sigma$ on $[0, 1]$ is a $U$-fuzzy $t$-subnorm of $([0, 1], T)$ if and only if the mapping $f:[-\infty, \infty] \to [-\infty, \infty]: $
	$$
	f=(-h) \circ \sigma \circ t^{[-1]}
	$$
	is subadditive.
\end{proposition}
\begin{proof}
	According to Definition~\ref{d:U-fuzzy_t-subnorm_T}, $\sigma$ is a $U$-fuzzy $t$-subnorm of $([0, 1], T)$ if and only if for all $x, y \in [0, 1]$
	$$
	U(\sigma(x), \sigma(y)) \leq \sigma\left(T(x, y)\right).
	$$
	Further, it can be denoted as
	$$
	h^{[-1]}(h(\sigma(x))+h(\sigma(y))) \leq \sigma\left(t^{[-1]}(t(x)+t(y))\right).
	$$
	Since
	$$
	(-h)^{[-1]}((-h)(\sigma(x))+(-h)(\sigma(y)))=h^{[-1]}\left(h(\sigma(x))+h(\sigma(y))\right)
	,$$ then
	$$
	(-h)^{[-1]}((-h)(\sigma(x))+(-h)(\sigma(y))) \leq \sigma\left(t^{[-1]}(t(x)+t(y))\right),
	$$
	which is equivalent to
	$$
	(-h)(\sigma(x))+(-h)(\sigma(y)) \geq (-h)\left(\sigma\left(t^{[-1]}(t(x)+t(y))\right)\right).
	$$
	Let $t(x)=a$ and $t(y)=b$,
	\begin{align*}
	(-h)\left(\sigma\left(t^{[-1]}(a)\right)\right)+(-h)\left(\sigma\left(t^{[-1]}(b)\right)\right) \geq h\left(\sigma\left(t^{[-1]}(a+b)\right)\right).
	\end{align*}	
\end{proof}

\begin{corollary}
	Let $f : [0, \infty] \rightarrow[-\infty, \infty]$ be a subadditive mapping and representable uninorms $U$ and continuous Archimedean $t$-norms $T$ have additive generators $h$ and $t$, respectively. Then
	$$
	(-h)^{[-1]} \circ f \circ t
	$$
	is a $U$-fuzzy $t$-subnorm of $([0, 1], T)$.
\end{corollary}
\begin{example}\label{3.4}
	Let $T_L$ be the {\L}ukasiewicz $t$-norm with additive generator $t$, where $t(x)=1-x$. $U_p$ is a uninorm with additive generator $h$,
	$$	 h(x)=\left\{
	\begin{aligned}
	&\textnormal{ln}(2x), &0\leq x <\frac{1}{2}, \\
	&-\textnormal{ln}(-2x+2), & \frac{1}{2} \leq x \leq 1,
	\end{aligned}
	\right.
	$$ and $ f(x)=\sqrt{x}$. At this moment, $u(x)=h^{-1}\circ f \circ t(x)$ is the $U_p$-fuzzy $t$-subnorm of $T_L$ and $U_p$ as shown in Figure. \ref{fig1}.
	\begin{figure}[h]
		\centering
		\includegraphics[width=8cm,height=6cm]{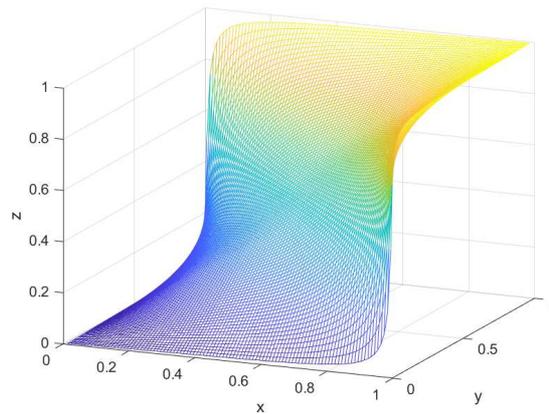}
		\caption{Three-dimensional image of $U_p(x,y)=h^{-1}(h(x)+h(y))$}
		\label{fig1}
	\end{figure}
\end{example}

\begin{example}
	Similarly, if we substitute the addition generator of  $U_p$ in Example \ref{3.4} with the following form,
	$$	 h(x)=\left\{
	\begin{aligned}
	&1-\frac{1}{2x}, &0\leq x \leq\frac{1}{2}, \\
	&-\frac{1}{2(x-1)}-1, & \frac{1}{2}< x \leq 1.
	\end{aligned}
	\right.
	$$
	Then $u(x)=h^{-1}\circ f \circ t(x)$ is the $U_p$-fuzzy $t$-subnorm of $T_L$ and $U_p$ as shown in Figure. \ref{fig2}.
\end{example}

\begin{figure}[h]
	\centering
	\includegraphics[width=8cm,height=6cm]{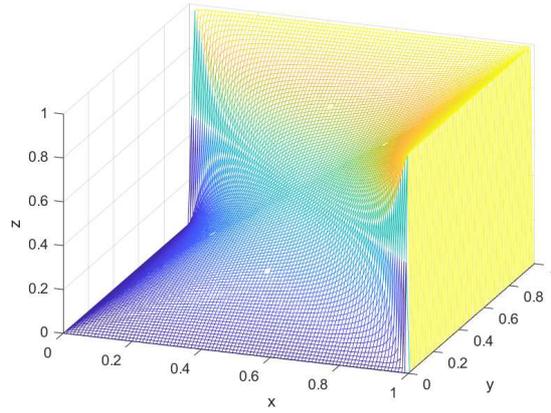}
	\caption{Three-dimensional image of $U_p(x,y)=h^{-1}(h(x)+h(y))$}
	\label{fig2}
\end{figure}

\begin{proposition}
	Let $U$ be representable uninorms and $S$ be continuous Archimedean $t$-conorms. $h$ and $s$ are additive generators of $U$ and $S$, respectively. A fuzzy subset $\sigma$ of $[0,1]$ is a $U$-fuzzy $t$-subconorm of $([0, 1], S)$ if and only if the mapping $	f:[0, \infty] \rightarrow[-\infty, \infty]:$
	\begin{align*}
	f=h \circ \sigma \circ s^{[-1]}
	\end{align*}
	is subadditive.
\end{proposition}

\begin{corollary}
	Let $f:[0, \infty] \rightarrow[-\infty, \infty]$ be a subadditive mapping, representable uninorms $U$ and continuous Archimedean $t$-conorms $S$ have additive generators $h$ and $s$, respectively. Then
	$$
	(-h)^{[-1]} \circ f \circ s
	$$
	is a $U$-fuzzy $t$-subconorm of $([0, 1], S)$.	
\end{corollary}

\subsection{$F$-fuzzy submonoid}
Similar to section \ref{$U$-fuzzy submonoid}, when the aggregation operations take nullnorms, we have the following conclusions.
\begin{definition}
	Let $F$ be a nullnorm, $(M,\circ)$ a monoid with identity element $e$ and $\sigma$ a fuzzy subset of $M$. $\sigma$ is an $F$-fuzzy submonoid of $M$ if and only if $\sigma$ satisfies the following conditions:
	\begin{enumerate}[(1)]
		\item $F(\sigma(x), \sigma(y))\leq \sigma(x\circ y)$, for all $x,y\in[0,1]$;
		\item $\sigma(e)=1$.
	\end{enumerate}
\end{definition}

\begin{definition}
	Let $(M,\circ)$ be a monoid and $\sigma$ be an $F$-fuzzy submonoid of $M$. Then the core $H$ of $\sigma$(i.e., the set of elements $x$ of $M$, such that $\sigma(x)=1$) is a submonoid of $M$.
\end{definition}

\begin{proof}
	The identity element of $H$ obviously existed and the associativity is inherited.
	Let $x,y \in H$.
	$$ 1=F(\sigma(x),\sigma(y))\leq\sigma(x\circ y).$$
	Therefore, $x \circ y \in H$.
\end{proof}

\begin{definition}
	A discrete nullnorm is a submonoid of a nullnorm contanining $0$ and $1$.
\end{definition}
Let $S$ and $T$ be a $t$-conorm and a $t$-norm, respectively. The nullnorm $F=\langle S,k,T \rangle$ with absorbing element $k$ as follows
\begin{equation*}
	F(x,y)=\left\{
	\begin{aligned}
	&kS(\frac{x}{k},\frac{y}{k}), &(x,y)\in[0,k]^2\\
	&k + (1-k)T(\frac{x-k}{1-k},\frac{y-k}{1-k}), &(x,y)\in (k,1]^2\\
	&k,	&otherwise.
	\end{aligned}
	\right.
\end{equation*}

\begin{example}
	Let $F_L $ be a nullnorm where $F_L=\langle S_L,k,T_L \rangle, L_{n,m}=\{0,\frac{k}{n},\cdots,k,k+\frac{1-k}{m},\cdots,1 \}$. Then $L_{n,m}$ is a discrete uninorm of $F_L$.	
\end{example}

\begin{definition}
	Let $F$ and $T$ be a nullnorm and a $t$-norm, respectively. An $F$-fuzzy submonoid of $([0,1],T)$ will be called an $F$-fuzzy $t$-subnorm of $T$.
\end{definition}
\begin{definition}
	Let $F$ and $S$ be a uninorm and a $t$-conorm, respectively. An $F$-fuzzy submonoid of $([0,1],S)$ will be called an $F$-fuzzy $t$-subconorm of $S$.
\end{definition}

\begin{proposition}
	If $\sigma$ is an $F$-fuzzy submonoid of $M$ where $F$ is a nullnorm with absorbing element $k$ and $M$ is a monoid with identity element $e$, then $\sigma(x)\geq k$ for any $x\in[0,1]$.
\end{proposition}

\begin{proof}
	If $\sigma$ is an $F$-fuzzy submonoid of $M$, then we have $$F(\sigma(x),\sigma(y))\leq \sigma(x\circ y)\ and\ \sigma(e)=1.$$
	Let $y=e$. We get that
	$$k=F(\sigma(x),k)\leq F(\sigma(x),\sigma(e))\leq \sigma(x).$$
\end{proof}

It is easy to get the following corollary when $M$ takes the $t$-norm $T$ and $t$-conorm $S$.

\begin{corollary}
	If $\sigma$ is an $F$-fuzzy $t$-subnorm of $T$ where $F$ is a nullnorm with absorbing element $k$ and $T$ is a $t$-norm, then $\sigma(x)\geq k$ for any $x\in[0,1]$.
\end{corollary}

\begin{corollary}
	If $\sigma$ is an $F$-fuzzy $t$-subconorm of $S$ where $F$ is a nullnorm with absorbing element $k$ and $S$ is a $t$-conorm, then $\sigma(x)\geq k$ for any $x\in[0,1]$.
\end{corollary}

Let $S$ and $T$ be a $t$-conorm and a $t$-norm, $F_M$ denoted the nullnorms $F$ with absorbing element $k$ as follows
\begin{equation*}
	F_M(x,y)=\left\{
	\begin{aligned}
	&kS(\frac{x}{k},\frac{y}{k}), &(x,y)\in[0,k]^2\\
	&k+(1-k)T_M(\frac{x-k}{1-k},\frac{y-k}{1-k}),&(x,y)\in (k,1]^2\\
	&k,	&otherwise.
	\end{aligned}
	\right.
\end{equation*}

\begin{proposition}\label{F}
	A fuzzy subset $\sigma$ is an $F_M$-fuzzy $t$-subnorm of $([0,1],T_M)$ where $F_M$ is a nullnorm with absorbing element $k$ if and only if $\sigma(1)=1$, $\sigma(x)\geq k $ for any $x\in[0,1]$.
\end{proposition}
\begin{proof}
	Firstly, if  $\sigma$ is an $F_M$-fuzzy $t$-subnorm of $([0,1],T_{\textnormal{min}})$, then we obviously have  $\sigma(1)=1$ and $\sigma(x)\geq k$ for any $x\in[0,1]$.
	
	Conversely, let $\sigma(1)=1$, $\sigma(x)\geq k $ for any $x\in[0,1]$. Then the discussion will be divided into three cases:
	\begin{enumerate}[(1)]
		\item
		If $(x,y)\in [0,k]^2$,
		\begin{align*}
		F_M(\sigma(x),\sigma(y))& =kS(\frac{\sigma(x)}{k},\frac{\sigma(y)}{k})\\
		& \leq k \\
		& \leq \sigma(T_M(x,y)).
		\end{align*}
		\item
		If $(x,y)\in [k,1]^2$,
		\begin{align*}
		F_M(\sigma(x),\sigma(y))& =	k+(1-k)T_M(\frac{\sigma(x)-k}{1-k},\frac{\sigma(y)-k}{1-k})\\ &=\textnormal{min}(\sigma(x),\sigma(y))\\
		& \leq \sigma(T_M(x,y)).
		\end{align*}	
		\item In the other cases, $ F_M(\sigma(x),\sigma(y))\leq \sigma(\sigma(x),\sigma(y))$ clearly holds.
	\end{enumerate}
\end{proof}

Then the dual conclusions about $F_M$-fuzzy $t$-subnorm of $([0,1],S_M)$ can be obtained.
\begin{proposition}
	A fuzzy subset $\sigma$ is an $F_M$-fuzzy $t$-subnorm of $([0,1],S_M)$ where $F_M$ is a nullnorm with absorbing element $k$ if and only if $\sigma(0)=1$, $\sigma(x)\geq k $ for any $x\in[0,1]$.
\end{proposition}
\begin{proof}
	It can be proven in a similar way as Proposition~\ref{F}.
\end{proof}

\subsection{Lattice value fuzzy submonoid}\label{Lattice}
Let $(L,\lor,\land,\leq)$ be a bounded lattice with maximum element $\mathbf{1}$ and minimum element $\mathbf{0}$. A map $A:M \rightarrow L $ will be called $L$-fuzzy set of $M$ and the family of all $L$-fuzzy set of $M$ is denoted $\mathscr{F}_L(M)$.

We can generalize the definition of fuzzy $t$-subnorm and $t$-subconorm onto lattice valued fuzzy sets.
\begin{definition}
	Let $(L,\lor,\land,\leq)$ be a bounded lattice with maximum element $\mathbf{1}$ and minimum element $\mathbf{0}$, $(M,\circ)$ a monoid with identity element $e$ and  $\sigma$ a $L$-fuzzy set of $M$. Then $\sigma$ is a $\land$-fuzzy submonoid of $M$ if and only if $\sigma$ satisfies the following conditions:
	\begin{enumerate}[(1)]
		\item $(\sigma(x)\land\sigma(y)) \leq \sigma(x\circ y)$, for all $x,y\in M$;
		\item $\sigma(e)=\mathbf{1}$.
	\end{enumerate}
\end{definition}

\begin{proposition}
	Let $\sigma$ be a $\land$-fuzzy submonoid. The core of $\sigma$ (i.e., the set of elements $ x$ of $M$ such that $\sigma(x)=\mathbf{1}$) is a submonoid of $M$.
\end{proposition}

\begin{proof}
	The identity element of $H$ obviously existed and the associativity is inherited.
	Then for all $x,y\in H$,
	$$ \mathbf{1} \leq (\sigma(x)\land\sigma(y)) \leq \sigma(x\circ y), $$
	so $x\circ y \in H$.
	
\end{proof}

\begin{definition}
	Let $(L,\lor,\land,\leq)$ and $T$ be a bounded lattice and $t$-norm, respectively. A $\land$-fuzzy submonoid of $([0,1],T)$ is also called $\land$-fuzzy $t$-subnorm of $([0,1],T)$.
\end{definition}

\begin{definition}
	Let $(L,\lor,\land,\leq)$ and $S$ be a bounded lattice and $t$-conorm, respectively. A $\land$-fuzzy submonoid of $([0,1],S)$ is also called $\land$-fuzzy $t$-subconorm of $([0,1],S)$.
\end{definition}

By using the same generalization method, the operation $\lor$ also has above definitions and propositions.

\section{Vague monoids}\label{Vague monoids}
To explore the homomorphisms between  $A$-vague monoid and $A$-fuzzy monoid,  the concepts of $A$-vague binary operations and $A$-vague monoids are put forward below.

\subsection{Vague monoids by aggregation operators}
The fuzzification of equivalence relations are fuzzy relations called $T$-fuzzy equivalences or $T$ -indistinguishability
operators. They are essential in fuzzy logic and have been widely studied \cite{J.RecasensIndistinguishability2011Operators}. In the following, we extend the definition of $T$-fuzzy equivalences to binary aggregation operators.
\begin{definition}
	Let $A$ be a binary aggregation operator and $X$ a set. A fuzzy relation $E : X\times X \rightarrow [0, 1] $ is an $A$-indistinguishability operator, which satisfies the following conditions for all $x, y, z \in X $:
	\begin{enumerate}[(1)]
		\item  $E(x, x) = 1$;
		\item  $E(x, y) = E(y, x);$
		\item  $A(E(x, y), E(y, z)) \le E(x, z)$.
	\end{enumerate}
	If $E(x, y) = 1$ implies $x = y,$ then it is said that $E$ separates points.
\end{definition}

\begin{definition}
	A fuzzy binary operation on a set $M$ is a mapping $\tilde{\circ}  : M \times M \times M \rightarrow [0, 1]$, where $\tilde{\circ}(x, y, z)$ is interpreted as the degree in which $z$ is $x\circ y$.
\end{definition}

Similar to the method of \cite{D.BoixaderVagueandfuzzy2021}, the vague binary operation based on aggregation functions is given below.
\begin{definition}\label{fuzzy_vague}
	Let $E$ be an $A$-indistinguishability operator on $M$. An $A$-vague binary operation is a fuzzy binary operation on $M$ when it satisfies the following conditions for all $x, y, z, x' , y' , z' \in M$:
		\begin{enumerate}[(1)]
		\item  $A(\tilde{\circ}(x, y, z),E(x, x'),E(y, y'),E(z, z')) \leq \tilde{\circ}(x' , y' , z ');$
		\item  $A(\tilde{\circ}(x, y, z), \tilde{\circ}(x, y, z')) \leq E(z, z');$
		\item   For all $ x, y \in M$, there exists $z\in M$ such that $\tilde{\circ} (x, y, z) = 1$.
	\end{enumerate}	
\end{definition}

\begin{proposition}
	Let $E$ be an $A$-indistinguishability operator on $M$ separating points and $\tilde{\circ}$ a vague binary operation on
	$M$. Then $z$ in Definition~\ref{fuzzy_vague} \textnormal{(3)} is unique.
\end{proposition}

\begin{proof}
	Let $z, z'\in M$ satisfy $\tilde{\circ}(x, y, z) = 1$ and  $\tilde{\circ}(x, y, z') = 1$. From Definition~\ref{fuzzy_vague}, one can conclude that
	$$ 1 = A( \tilde{\circ}(x, y, z),  \tilde{\circ}(x, y, z' )) \leq E(z, z').$$ Hence, $E(z, z') = 1$. Since $E$ separates points, we have $z = z'$.
\end{proof}
\begin{definition}
	Let $\tilde{\circ}$ be an $A$-vague binary operation on $M$ with respect to an $A$-indistinguishability operator $E$ on $M$.
	Then $(M, \tilde{\circ})$ is an $A$-vague monoid if and only if it satisfies the following properties:
	\begin{enumerate}[(1)]
		\item  Associativity:  $$A(\tilde{\circ}(y,z,d),\tilde{\circ}(x,d,m),\tilde{\circ}(x,y,q),\tilde{\circ}(q,z,w)) \leq E(m,w),$$
		where $ x, y, z, d, m, q \in M$.
		\item Identity: There exists identity element $e \in M$ such that
		\begin{align*}
		\tilde{\circ}(e,x,x)=1\ \textnormal{and}\  \tilde{\circ}(x,e,x)= 1\ \textnormal{for all}\ x \in M.
		\end{align*}
	\end{enumerate}
\end{definition}
\begin{definition}
	An $A$-vague monoid is commutative if and only if
	for all $x,y,m,w \in M$, $$A(\tilde{\circ}(x,y,m), \tilde{\circ}(y,x,w)) \leq E(m,w).$$
	
\end{definition}
Next, we give a sufficient condition for the uniqueness of identity element of $A$-vague monoid $(m,\tilde{\circ})$.
\begin{proposition}
	Let $A$ be an aggregation operator, $E$ an $A$-indistinguishability operator separating points on set $M$ and $(M,\tilde{\circ})$ an $A$-vague monoid. Then the identity element is unique.
\end{proposition}
\begin{proof}
	Let $e$ and $e'$
 be two identity elements of $M$. Then, since $e$ is an identity element,
	we have that $$\tilde{\circ}(e,e',e') = 1.$$
	Since $e'$
 is an identity element, we have that
	$$\tilde{\circ}(e,e',e) = 1.$$
	From Definition~\ref{fuzzy_vague},
	$$1 = A(\tilde{\circ}(e,e',e'), \tilde{\circ}(e,e',e)) \leq E(e,e').$$
	So $E(e, e') = 1$ and $e = e'$
 because $E$ separates points.
\end{proof}

\begin{definition}
	Let $\circ$ be a binary operation on $M$ and $E$ an $A$-indistinguishability operator on $M$. Then $E$ is regular with respect to $\circ$ if and only if for all $x, y, z \in M$,
	$$E(x,y) \leq E(x\circ z,y \circ z)$$
	and
	$$E(x,y) \leq E(z\circ x,z \circ y).$$
\end{definition}

Under the condition of regular $A$-indistinguishability operators, the definition of the fuzzy mapping $\tilde{\circ}$ can be given and further the $A$-vague monoid is obtained.

\begin{proposition}
	Let $E$ be a regular $A$-indistinguishability operator on $M$ with respect to a binary operation $\circ$ on $M$.
	\begin{enumerate}
		\item[(1)] The fuzzy mapping $\tilde{\circ} : M \times M \times M \rightarrow [0, 1]$ defined by
		$$\tilde{\circ}(x,y,z) = E(x \circ y,z), \forall x,y,z \in M,$$
		is an $A$-vague binary operation on $M$.
		\item[(2)] If $(M, \circ)$ is a monoid, then $(M,\tilde{\circ})$ is an $A$-vague monoid.
	\end{enumerate}
\end{proposition}
\begin{proof}
	\begin{enumerate}[(1)]
		\item
		Firstly, we prove that $\tilde{\circ}$ satisfies the properties of Definition~\ref{fuzzy_vague}. It follows that for all $x, y, z, x', y', z'\in M$.	
		
			At the beginning, let's verify that  $A(\tilde{\circ}(x, y, z),E(x, x'),E(y, y'),E(z, z')) \leq \tilde{\circ}(x' , y' , z ')$.
			\begin{align*}
			&A(\tilde{\circ}(x,y,z),E(x,x'),E(y,y'),E(z,z'))\\
			&= A(E(x\circ y,z),E(x,x'),E(y,y'),E(z,z'))\\
			&\leq A(E(x\circ y,z'),E(x,x'),E(y,y))\\
			&\leq A(E(x \circ y,z'),E(x \circ y,x'\circ y),E(y,y'))\\
			&\leq A(E(x'\circ y,z'),E(y,y'))\\
			&\leq A(E(x'\circ y,z'),E(x'\circ y,x'\circ y'))\\
			&\leq  E(x'\circ y',z') \\
			&= \tilde{\circ}(x',y',z').		
			\end{align*}
	
	Then, $A(\tilde{\circ}(x,y,z), \tilde{\circ}(x,y,z'))
		=A(E(x \circ y,z),E(x\circ y,z'))
		\leq E(z,z')$
		
		Furthermore, for all $ x, y\in M$ we can consider $z = x \circ y$. Then
		$$\tilde{\circ}(x,y,z) = E(x\circ y,x\circ y) = 1.$$
		
		\item
		In the following,  let's verify that $(M,\tilde{\circ})$ has associativity and identity element.
		
		By the associativity of $\circ$ and the regularity of $E$, for all $x, y, z, d, m, q, w \in M$, we have that
		\begin{align*}
		&A (\tilde{\circ}(y,z,d),\tilde{\circ} (x,d,m), \tilde{\circ}(x,y,q), \tilde{\circ}(q,z,w))\\
		&= A (E(y\circ z,d),E(x\circ d,m),E(x\circ y,q),E(q \circ z,w))\\
		&\leq A (E(x \circ (y \circ z),x \circ d),E(x \circ d,m),E((x \circ y) \circ z,q \circ z),E(q \circ z,w))\\
		&\leq A (E(x \circ (y \circ z),m),E((x \circ y) \circ z,w))\\
		&\leq E(m,w).
		\end{align*}
		Suppose that $e$ be the identity element of $M$, it holds that
		$$\tilde{\circ}(x,e,x) = E(x \circ e,x) = E(x,x) = 1,$$
		$$\tilde{\circ}(e,x,x) = E(e \circ x,x) = E(x,x) = 1,$$
		then $e$ is the identity element of $\tilde{\circ}$.
		Hence, $(M,\tilde{\circ})$ is an $A$-vague monoid.
		\end{enumerate}
	\end{proof}
    Conversely, a monoid $(M,\circ)$ can be obtained from a vague monoid $(M, \tilde{\circ})$.
\begin{proposition}
Let $(M, \tilde{\circ})$ be a vague monoid with respect to an $A$-indistinguishability operator $E$ separating points. Then $(M, \circ)$ is a monoid where $x\circ y$ is the unique $z\in M$ such that $ \tilde{\circ}(x, y, z) = 1$.
\end{proposition}

\begin{proof}
	First to verify the associativity,
		\begin{align*}
		1& =A (\tilde{\circ}(y,z,y \circ z), \tilde{\circ}(x,y \circ z,x \circ (y \circ z)), \tilde{\circ}(x,y,x \circ y), \tilde{\circ}(x \circ y,z,(x \circ y) \circ z))\\
	   &\leq E(x \circ (y \circ z),(x \circ y) \circ z)
	\end{align*}
	Therefore, $ x \circ (y \circ z) = (x \circ y) \circ z$.

	Further, we can obtain that
	$$	1 = A(\tilde{\circ}(x,e,x),\tilde{\circ}(x,e,x\circ e))\le E(x \circ e,x),$$
	$$		1 = A(\tilde{\circ}(e,x,x),\tilde{\circ}(e,x,e\circ x))\le E(e\circ x,x),$$

	Hence, $E(x \circ e,x)=1$ and $E(e \circ x,x)=1$, that is $ x \circ e = x$ and $e \circ x = x$, then $e$ is the identity element of $(M, \circ)$.
\end{proof}

\begin{definition}
	Let $(M, \tilde{\circ})$ be a vague monoid with respect to an $A$-indistinguishability operator $E$ separating points. Then $(M,\circ)$ is the monoid associated to the vague monoid $(M,\tilde{\circ})$.
\end{definition}

 Further, if $A$ is an aggregation operator and $(M,\circ)$ is a monoid, then there exist bijective maps between their $A$-vague monoids and regular $A$-indistinguishability operators and they can represent each other:

\begin{itemize}
	\item $\tilde{\circ}(x,y,z)=E(x\circ y,z);$
	\item $E(x,y)=\tilde{\circ}(x,e,y).$
\end{itemize}

Regarding commutativity, the connection between $(M,\tilde{\circ})$ and $(M, \circ)$ is given.

\begin{proposition}
	Let $E$ be an $A$-indistinguishability operator separating points, $(M, \tilde{\circ})$ an $A$-vague monoid and $(M, \circ)$ its associated monoids $(x \circ y = z$ if and only if $\tilde{\circ}(x, y, z) = 1)$. Then $(M, \tilde{\circ})$ is commutative if and only if
	$(M, \circ)$ is commutative.
\end{proposition}
\begin{proof}
	Firstly, if $(M, \tilde{\circ})$ is commutative, we have that
	$$  1= A(\tilde{\circ}(x,y,x \circ y),\tilde{\circ}(y,x,y\circ x)) \leq E(x \circ y,y \circ x).$$
	From $E$ separating points and $E(x \circ y, y \circ x) = 1$, we have $x\circ y=y\circ x$.
	Inversely,  if	$(M, \circ)$ is commutative, we have
	\begin{align*}
	A (\tilde{\circ}(x,y,z), \tilde{\circ}(y,x,z'))
	&= A (E(x \circ y,z),E(y \circ x,z'))\\
	&= A(E(x \circ y,z),E(x \circ y,z'))\\
	&\leq E(z,z').
	\end{align*}
So $(M, \tilde{\circ})$ is commutative.
\end{proof}

 The relation between $T$-vague monoid and $T$-fuzzy monoid have been studied in \cite{D.BoixaderVagueandfuzzy2021}. In the following, the homomorphisms between them is extended to $A$-vague monoid and $A$-fuzzy monoid.

\begin{definition}
	Let $(M, \tilde{\circ})$ and $(N,\tilde{\bullet})$ be two $A$-vague monoids with respect to the $A$-indistinguishability operators $E$ and $F$, respectively. A map $f:M \rightarrow N$ is a homomorphism from $M$ onto $N$ if and only if
	$\tilde{\circ}(x,y,z) \leq \tilde{\bullet}(f(x),f(y),f(z))\ \textnormal{for all}  \ x,y,z \in M.$
\end{definition}
\begin{lemma}
	Let $(M, \tilde{\circ})$ and $(N,\tilde{\bullet})$ be two $A$-vague monoids with respect to $A$-indistinguishability operators $E$ and $F$ respectively and  $f:M \rightarrow N$ a homomorphism from $M$ onto $N$. If $e$ is the identity element of $M$, then $f(e)$ is the identity element of $N$.
\end{lemma}
\begin{proof}
Since
	\begin{align*}
	&1= \tilde{\circ}(x,e,x) \leq \tilde{\bullet}(f (x),f (e),f (x)),\\
	&1 = \tilde{\circ}(e,x,x) \leq \tilde{\bullet}(f (e),f (x),f (x)).
	\end{align*}
	Hence, $\tilde{\bullet}(f (x), f (e), f (x)) =\tilde{\bullet}(f (e), f (x), f (x)) = 1$ and $f(e)$ is the identity element of $N$.
\end{proof}
\begin{proposition}
		Let $(M, \tilde{\circ})$ and $(N,\tilde{\bullet})$ be two $A$-vague monoids with respect to $A$-indistinguishability operators $E$ and $F$ respectively such that $E\leq F$. Then the identity map $\mathrm{id}: M \rightarrow M$ is a homomorphism from $(M,\tilde{\circ})$ onto $(M,\tilde{\bullet})$.
\end{proposition}

\begin{proof}
	$$\tilde{\circ}(x,y,z)=E(x\circ y,z)\leq F(x\circ y,z)\leq \tilde{\bullet}(x,y,z).$$
\end{proof}
A crisp monoid $(M,\circ)$ is an $A$-vague monoid when $x\circ y=z$ then $\circ(x,y,z)=1$ and $0$ otherwise  for all $x, y, z\in M$.
\begin{corollary}
	Let  $(M, \tilde{\circ})$ be an $A$-vague monoid with respect to an $A$-indistinguishability operator $E$ separating points. Then the identity map $\mathrm{id}: M \rightarrow M$ is a homomorphism from $(M,\circ)$ onto $(M,\tilde{\circ})$.
\end{corollary}
\begin{definition}
	Let $f:M \rightarrow N$ be a homomorphism from $(M,\tilde{\circ})$ onto $(N,\tilde{\bullet})$. The kernel of $f$ is the fuzzy subset $\sigma$
	of $M$ defined by $\sigma(x)= E(f(x),e')$ for all $x \in M  $ where $e'$ is the identity element of $(N,\tilde{\bullet})$.
\end{definition}
\begin{proposition}
Let  $(M, \tilde{\circ})$ be an $A$-vague monoid with respect to an $A$-indistinguishability operator $E$ separating points. The kernel of the identity map $\mathrm{id}: M \rightarrow M$ is an $A$-fuzzy submonoid of $(M, \circ)$.
\end{proposition}
\begin{proof}
	Firstly, we have that $ker$id$(e) =E(e,e)=1$. On the other side, it holds that
	\begin{align*}
	A(ker\textnormal{id}(x),ker\textnormal{id}(y))&=A(E(x,e),E(y,e))\\
						&\leq A(E(x\circ y,y),E(y,e))\\
						&\leq E(x\circ y,e)\\
						&=ker\textnormal{id}(x\circ y).
	\end{align*}
\end{proof}

\subsection{Vague monoid by uninorms and nullnorms}
Since uninorm and nullnorm are special binary aggregation operators. Then we can obtain same conclusions and just set out them as follow.

\subsection{Uninorm vague monoid}

\begin{definition}
	Let $U$ be a uninorm and $X$ a set. A fuzzy relation $E : X\times X \rightarrow [0, 1] $ is a $U$-indistinguishability operator, which satisfies the following conditions
    for all $x, y, z \in X $:
	\begin{enumerate}[(1)]
		\item  $E(x, x) = 1 $;
		\item $ E(x, y) = E(y, x)$;
		\item $ U(E(x, y),E(y, z)) \leq E(x, z).$
	\end{enumerate}
	If $E(x, y) = 1$ implies $x = y,$ then it is said that $E$ separates points.
\end{definition}

\begin{definition}
	Let $E$ be a $U$-indistinguishability operator on $M$. A $U$-vague binary operation $\tilde{\circ}: M \times M \times M \rightarrow [0, 1]$
	is a fuzzy binary operation on $M$ when it satisfies the following conditions for all $x, y, z, x' , y' , z' \in M$:
	\begin{enumerate}[(1)]
		\item  $U(\tilde{\circ}(x, y, z),E(x, x'),E(y, y'),E(z, z')) \leq \tilde{\circ}(x' , y' , z ');$
		\item  $U(\tilde{\circ}(x, y, z), \tilde{\circ}(x, y, z')) \leq E(z, z');$
		\item
		 For all $ x, y \in M$, there exists $z\in M$ such that $\tilde{\circ} (x, y, z) = 1$.
	\end{enumerate}
	
\end{definition}
\begin{proposition}
	Let E be a $U$-indistinguishability operator on $M$ separating points and $\tilde{\circ}$ a vague binary operation on
	$M$. Then the $z$ in Definition~\ref{fuzzy_vague} \textnormal{(3)} is unique.
\end{proposition}

\begin{definition}
	Let $\tilde{\circ}$ be a $U$-vague binary operation on $M$ with respect to a $U$-indistinguishability operator $E$ on $M$.
	Then $(M, \tilde{\circ})$ is a $U$-vague monoid if and only if it satisfies the following properties:
	\begin{enumerate}[(1)]
		\item  Associativity:  $$U(\tilde{\circ}(y,z,d),\tilde{\circ}(x,d,m),\tilde{\circ}(x,y,q),\tilde{\circ}(q,z,w)) \leq E(m,w),$$
		 where $x, y, z, d, m, q \in M$;
		\item Identity:	There exists identity element $e \in M$ such that $$\tilde{\circ}(e,x,x)=1\ \textnormal{and } \tilde{\circ}(x,e,x) = 1\ \textnormal{for all }x\in M.$$	
	\end{enumerate}
\end{definition}
\begin{definition}
	A $U$-vague monoid is commutative if and only if
	for all $x,y,m,w \in M$, $$U(\tilde{\circ}(x,y,m), \tilde{\circ}(y,x,w)))\leq E(m,w).$$	
\end{definition}

Next, we give a sufficient condition for the uniqueness of identity element of $U$-vague monoid $(m,\tilde{\circ})$.

\begin{proposition}
	Let $U$ be a uninorm, $E$ a $U$-indistinguishability operator separating points on set $M$ and $(M,\tilde{\circ})$ a $U$-vague monoid. Then the identity element is unique.
\end{proposition}

\begin{definition}
	Let $\circ$ be a binary operation on $M$, and $E$ a $U$-indistinguishability operator on $M$. $E$ is regular with respect to $\circ$ if and only if for all $x, y, z \in M$,
	$$E(x,y) \leq E(x\circ z,y \circ z)$$
	and
	$$E(x,y) \leq E(z\circ x,z \circ y).$$
\end{definition}
Under the condition of regular $U$-indistinguishability operators, the definition of the fuzzy mapping $\tilde{\circ}$ can be given and further the $U$-vague monoid is obtained.

\begin{proposition}
	Let $E$ be a regular $U$-indistinguishability operator on $M$ with respect to a binary operation $\circ$ on $M$.
	\begin{enumerate}[(1)]
		\item
		 The fuzzy mapping $\tilde{\circ} : M \times M \times M \rightarrow [0, 1]$ defined by
		$$\tilde{\circ}(x,y,z) = E(x \circ y,z), \forall x,y,z \in M,$$
		is a $U$-vague binary operation on $M$.
		\item
		 If $(M, \circ)$ is a monoid, then $(M,\tilde{\circ})$ is a $U$-vague monoid.
	\end{enumerate}
\end{proposition}

Conversely, a monoid $(M,\circ)$ can be obtained from a vague monoid $(M, \tilde{\circ})$.
\begin{proposition}
	Let $(M, \tilde{\circ})$ be a vague monoid with respect to a $U$-indistinguishability operator $E$ separating points. Then $(M, \circ)$ is a monoid where $x\circ y$ is the unique $z\in M$ such that $ \tilde{\circ}(x, y, z) = 1$.
\end{proposition}

\begin{definition}
Let $(M, \tilde{\circ})$ be a vague monoid with respect to a $U$-indistinguishability operator $E$ separating points. Then $(M,\circ)$ is the monoid associated to the vague monoid $(M,\tilde{\circ})$.
\end{definition}

Further, if $U$ is a uninorm and $(M,\circ)$ is a monoid, then there exist bijective maps between their $U$-vague monoids and regular $U$-indistinguishability operators and they can represent each other:


\begin{itemize}
	\item $\tilde{\circ}(x,y,z)=E(x\circ y,z);$
	\item $E(x,y)=\tilde{\circ}(x,e,y).$
\end{itemize}

Regarding commutativity, the connection between $(M,\tilde{\circ})$ and $(M, \circ)$ is given.

\begin{proposition}
	Let $E$ be a $U$-indistinguishability operator separating points, $(M, \tilde{\circ})$ a $U$-vague operation and $(M, \circ)$ its associated operation $(x \circ y = z$ if and only if $\tilde{\circ}(x, y, z) = 1)$. Then $(M, \tilde{\circ})$ is commutative if and only if
	$(M, \circ)$ is commutative.
\end{proposition}

\subsection{Nullnorm vague monoid}

\begin{definition}
	Let $F$ be a nullnorm and $X$ a set. A fuzzy relation $E : X\times X \rightarrow [0, 1] $ is an $F$-indistinguishability operator, which satisfies the following conditions for all $x, y, z  \in X $:
	\begin{enumerate}[(1)]
		\item  $E(x, x) = 1;$
		\item $ E(x, y) = E(y, x);$
		\item $ F(E(x, y),E(y, z)) \leq E(x, z).$
	\end{enumerate}
	If $E(x, y) = 1$ implies $x = y,$ then it is said that $E$ separates points.
\end{definition}

\begin{definition}
	Let $E$ be an $F$-indistinguishability operator on $M$. An $F$-vague binary operation $\tilde{\circ}: M \times M \times M \rightarrow [0, 1]$ is a fuzzy binary
	operation on $M$ when it satisfies the following conditions for all $x, y, z, x' , y' , z' \in M$:
	\begin{enumerate}[(1)]
		\item
		 $F(\tilde{\circ}(x, y, z),E(x, x'),E(y, y'),E(z, z')) \leq \tilde{\circ}(x' , y' , z ');$
		\item
		 $F(\tilde{\circ}(x, y, z), \tilde{\circ}(x, y, z')) \leq E(z, z');$
		\item
		 For all $ x, y \in M$ there exists $z\in M$ such that $\tilde{\circ} (x, y, z) = 1$.
	\end{enumerate}
	
\end{definition}
\begin{proposition}
	Let E be $F$-indistinguishability operator on $M$ separating points and $\tilde{\circ}$ a vague binary operation on
	$M$. Then the $z$ of Definition~\ref{fuzzy_vague} is unique.
\end{proposition}

\begin{definition}
	Let $\tilde{\circ}$ be an $F$-vague binary operation on $M$ with respect to an $F$-indistinguishability operator $E$ on $M$.
	Then $(M, \tilde{\circ})$ is an $F$-vague monoid if and only if it satisfies the following properties.
	\begin{enumerate}[(1)]
		\item  Associativity:   $$F(\tilde{\circ}(y,z,d),\tilde{\circ}(x,d,m),\tilde{\circ}(x,y,q),\tilde{\circ}(q,z,w)) \leq E(m,w),$$
		where $ x, y, z, d, m, q \in M$.
		\item Identity: There exists identity element $e \in M$ such that $$\tilde{\circ}(e,x,x)=1\ \textnormal{and}\ \tilde{\circ}(x,e,x) = 1\ \textnormal{	for\ all}\ x \in M.$$
	\end{enumerate}
\end{definition}
\begin{definition}
	An $F$-vague monoid is commutative if and only if
	for all $x,y,m,w \in M$, $$F(\tilde{\circ}(x,y,m), \tilde{\circ}(y,x,w))\leq E(m,w)).$$	
\end{definition}

Next, we give a sufficient condition for the uniqueness of identity element of $F$-vague monoid $(M,\tilde{\circ})$.

\begin{proposition}
	Let $F$ be nullnorm, $E$ an $F$-indistinguishability operator separating points on set $M$ and $(M,\tilde{\circ})$ an $F$-vague monoid. Then the identity element is unique.
\end{proposition}

\begin{definition}
	Let $\circ$ be a binary operation on $M$, and $E$ an $F$-indistinguishability operator on $M$. $E$ is regular with
	respect to $\circ$ if and only if for all $x, y, z \in M$,
	$$E(x,y) \leq E(x\circ z,y \circ z)$$
	and
	$$E(x,y) \leq E(z\circ x,z \circ y).$$
\end{definition}
Under the condition of regular $F$-indistinguishability operators, the definition of the fuzzy mapping $\tilde{\circ}$ can be given and further the $F$-vague monoid is obtained.

\begin{proposition}
	Let $E$ be a regular $F$-indistinguishability operator on $M$ with respect to a binary operation $\circ$ on $M$.
	\begin{enumerate}[(1)]
		\item
		The fuzzy mapping $\tilde{\circ} : M \times M \times M \rightarrow [0, 1]$ defined by
		$$\tilde{\circ}(x,y,z) = E(x \circ y,z), \forall x,y,z \in M,$$
		is an $F$-vague binary operation on $M$.
		\item
		 If $(M, \circ)$ is a monoid, then $(M,\tilde{\circ})$ is an $F$-vague monoid.
	\end{enumerate}
\end{proposition}

Conversely, a monoid $(M,\circ)$ can be obtained from a vague monoid on $(M, \tilde{\circ})$.
\begin{proposition}
	Let $(M, \tilde{\circ})$ be a vague monoid with respect to an $F$-indistinguishability operator $E$ separating points. Then $(M, \circ)$ is a monoid where $x\circ y$ is the unique $z\in M$ such that $ \tilde{\circ}(x, y, z) = 1$.
\end{proposition}

\begin{definition}
		Let $(M, \tilde{\circ})$ be a vague monoid with respect to an $F$-indistinguishability operator $E$ separating points. Then $(M,\circ)$ is the monoid associated to the vague monoid $(M,\tilde{\circ})$.
\end{definition}

Further, if $F$ is a nullnorm and $(M,\circ)$ is a monoid, then there exist bijective maps between their $F$-vague monoids and regular $F$-indistinguishability operators and they can represent each other:


\begin{itemize}
	\item $\tilde{\circ}(x,y,z)=E(x\circ y,z);$
	\item $E(x,y)=\tilde{\circ}(x,e,y).$
\end{itemize}

Regarding commutativity, the connection between $(M, \tilde{\circ})$ and $(M, \circ)$ is given.
\begin{proposition}
	Let $E$ be an $F$-indistinguishability operator separating points, $(M, \tilde{\circ})$ an $F$-vague operation and $(M, \circ)$ its associated operation $(x \circ y = z$ if and only if $\tilde{\circ}(x, y, z) = 1)$. Then $(M, \tilde{\circ})$ is commutative if and only if
	$(M, \circ)$ is commutative.
\end{proposition}

\section{Concluding remarks}\label{Concluding remarks}
Comparing the difference between using $t$-norms and aggregation functions as defining criteria, we can get the following points.
\begin{itemize}
	\item Notice that $t$-norm is a special aggregation function, substituting aggregate functions for $t$-norm can further generalize related properties. In the meantime, special aggregation functions such as uninorm and nullnorm can be brought in to expand the scope of this research.
	
	\item Despite losing control of the boundary conditions, the diversity of aggregation functions allows us to have more flexible options for dealing with related problems.
\end{itemize}

In the future, we intend to consider adding some properties such as archimedean and strictness to make more interesting and meaningful conclusions.

\section*{Acknowledgements}
   This research was supported by the National Natural Science Foundation of China (Grant no. 12101500), the Chinese Universities Scientific Fund (Grant no. 2452018054) and the College Students' Innovation and Entrepreneurship Training Program (Grant no. S202010712009).


\end{document}